\documentclass[12pt]{amsart}

\usepackage[english]{babel}
\usepackage{pgf}
\usepackage[utf8]{inputenc}
\usepackage{mathtools}
\usepackage{textcase}
\usepackage{hyperref}

\usepackage{array}
\usepackage{tabularx}
\usepackage{graphicx}
\usepackage{amsmath}
\usepackage{amssymb}
\usepackage{amsthm}

\usepackage{enumerate}

\usepackage{tikz}
\usetikzlibrary{shapes,shapes.geometric,arrows,fit,calc,positioning,arrows,automata,}

\allowdisplaybreaks

\newcommand{\N}{\mathbb{N}}

\newcommand{\R}{\mathbb{R}}
\newcommand{\Q}{\mathbb{Q}}

\newcommand{\cC}{\mathcal{C}}

\newcommand{\floor}[1]{\left\lfloor #1 \right\rfloor}

\newcommand{\abs}[1]{\left| #1 \right|}

\newcommand{\rb}[1]{\left( #1 \right)}

\newtheorem{theorem}{Theorem}[section]
\newtheorem*{theorem*}{Theorem}
\newtheorem{lemma}[theorem]{Lemma}

\newtheorem{proposition}[theorem]{Proposition}

\theoremstyle{definition}
\newtheorem{definition}[theorem]{Definition}

\theoremstyle{remark}
\newtheorem*{remark}{Remark}
\newtheorem*{example}{Example}

\usepackage{enumitem}

\newcommand{\set}[1]{\{#1\}}
\newcommand{\setm}[2]{\set{#1\mid#2}}
\newcommand{\VC}[1]{\mathrm{VC}\left( #1 \right)}

\begin{document}
 \selectlanguage{english}
 \title{The VC-Dimension of Axis-Parallel Boxes 
  on the Torus}
 \author{P. Gillibert}
 \email{gillibert.pierre@tuwien.ac.at}
\address{Institut f\"ur Diskrete Mathematik und Geometrie
TU Wien\\
Wiedner Hauptstr. 8--10\\
1040 Wien, Austria}
 
 \author{T. Lachmann}
 \email{thomas.lachmann@jku.at}
 \address{Institut f\"ur Finanzmathematik und Angewandte Zahlentheorie JKU Linz, Altenberger Straße 69
4040 Linz, Austria}
 
 \author{C. M\"ullner}
 \email{clemens.muellner@tuwien.ac.at}
\address{Institut f\"ur Diskrete Mathematik und Geometrie
TU Wien\\
Wiedner Hauptstr. 8--10\\
1040 Wien, Austria}
 
 \thanks{The first author is supported by by the Austrian Science Foundation FWF, project P32337.
 The second author is supported by the Austrian Science Foundation FWF, projects F5505-N26 and Y-901. The third author is supported by the Austrian Science Foundation FWF, project F5502-N26. The projects F5502-N26 and F5505-N26 are part of the Special Research Program ``Quasi Monte Carlo Methods: Theory and Applications''.}
 \maketitle
  \begin{abstract}
	We show in this paper that the VC-dimension of the family of $d$-dimensional axis-parallel boxes and cubes on the $d$-dimensional torus are both asymptotically $d \log_2(d)$.
	This is especially surprising as the VC-dimension usually grows linearly with $d$ in similar settings.
  \end{abstract}

	\section{Introduction}
The Vapnik-Chervonenkis-dimension (in short VC-dimension) is an important combinatorial concept and has interesting applications in different fields, such as discrepancy theory (see for example \cite{Matousek1999}, \cite{Matousek1993}, and \cite{Hinrichs2004}), connections to dispersion (see \cite{Rudolf2018}), measure theory (for literature see \cite{Dudley1984}, \cite{Dudley2014}, or \cite{vdVaart1996}), and machine learning (for literature see \cite{Mohri2018} or \cite{Shalev2014}). The last became more and more popular over the last years and is one of the foundation stone of artificial intelligence. 

The VC-dimension was introduced in~\cite{Vapnik1971} and can be used to measure the ability of a model to classify datasets \footnote{We also mention here the related problem of sample compression which is treated for example in \cite{Devroye1996}, \cite{Doliwa2014}, \cite{Moran2017}, \cite{Moran2016}, and \cite{Vapnik1998}.}. 
Given a dataset $C$ of $n$ points which can be labeled in $2^n$ different ways with $2$ labels (usually named \emph{positive} and \emph{negative}) and a \emph{hypothesis set} $H$. It is said that the hypothesis set $H$ shatters the point set $C$ if for any labeling of the $n$ points in $C$ there is a hypothesis $h\in H$ that separates the negative points from the positive points. The maximal number of points that can be shattered by a given hypothesis set $H$ is called the Vapnik-Chervonenkis-dimension $\VC{H}$ of $H$ (a more formal definition can be found in Section~\ref{sec_2}).
	
	Many classical examples of exactly computed VC-dimensions on the space $[0,1]^d$ are known whereas for these examples it is equivalent to consider $\R^d$.
	Already in~\cite{Vapnik1971} was established that the VC-dimension of half-spaces in dimension $d$ is $d+1$.
	As a direct consequence one finds that the VC-dimension of spheres in dimension $d$ is also $d+1$. 
	Another example are axis-parallel boxes in $[0,1]^d$ either anchored at the origin or not with corresponding VC-dimensions $d$ and $2d$ respectively.
	Furthermore, there is a new result by Despres~\cite{Despres} showing that the VC-dimension of cubes in $[0,1]^d$ is $\floor{\frac{3d+1}{2}}$.
	
	We denote by $\mathcal{B}^d_{\text{per}}$ (resp. $\mathcal{C}^d_{\text{per}})$ the set of $d$-dimensional axis-parallel boxes (resp. cubes) within the $d$-dimensional torus, see Section~\ref{sec_2} for precise definitions.
	With this we can state our main result. We are mainly interested in the case of axis-parallel boxes, but since the exact same method also works for cubes, we treat them as well.
	
	\begin{theorem}\label{th_main}
For $d$ sufficiently large we have
	
	\begin{align*}
		d(\log_2(d) - 4 \log_2(\log_2(d))) & \leq \VC{\mathcal{C}^d_{\text{per}}}  \\
		& \leq \VC{\mathcal{B}^d_{\text{per}}} \leq d(\log_2(d) + 3 \log_2(\log_2(d))).
	\end{align*}
	
	This shows in particular that $\VC{\mathcal{B}^d_{\text{per}}} \sim \VC{\mathcal{C}^d_{\text{per}}} \sim d \log_2(d)$.
\end{theorem}

\begin{remark}
For small dimensions, we know the exact values $\VC{\mathcal{B}^1_{\text{per}}}=3$, $\VC{\mathcal{B}^2_{\text{per}}}=6$, $\VC{\mathcal{B}^3_{\text{per}}}=11$, and the lower bound $\VC{\mathcal{B}^4_{\text{per}}}\geq 15$ which were determined using computer assistance (personal communication with Manfred Scheucher, TU Berlin).
\end{remark}

	This result is interesting for a few reasons. 
	First of all, the VC-dimension does not grow linearly in contrast to the (very similar looking) examples stated above.
	Secondly, there is an essential difference for the VC-dimension of boxes and cubes in $\R^d$ which seems to vanish when considering the $d$-dimensional Torus.
	Lastly, the example of boxes in $[0,1]^d$ can also be considered as the direct product of $d$ times the interval $[0,1]$. For such products there have been a few general lower bounds and some upper bounds. For example \cite{Dudley2014}, p. 192--200 contains results which implies the VC-dimension for boxes. Not a lot of general lower bounds are known for different hypothesis sets and usually there are some strong assumptions connected to the sets. Some interesting upper bounds are known, besides the ones in aforementioned sources, such as one given by van der Vaart and Wellner in \cite{vdVaart2009} including the so-called \emph{entropy} of the involved sets.
The shortage of such general bounds makes our result even more interesting since axis-parallel boxes on the torus do not fulfill any of the assumption usually used to calculate the VC-dimension of such hypothesis sets.
	
	The paper is structured as follows. In Section~\ref{sec_2} we give some basic definitions and discuss the VC-dimension of a very simple subclass of $\mathcal{B}^d_{\text{per}}$ called stripes.
	We give an upper bounds for $\VC{\mathcal{B}^d_{\text{per}}}$ in Section~\ref{sec_3} which solely relies on quite precise counting.
	The heart of the paper is Section~\ref{sec_4} where we establish the lower bound for $\VC{\mathcal{C}^d_{\text{per}}}$. This relies on a sophisticated scheme which allows to build configurations of points (based on the result for stripes) that can be shattered by $\mathcal{C}^d_{\text{per}}$.

\section{Basic definitions and stripes on the torus}\label{sec_2}
\subsection{Definitions}\label{subsec_def}
First we give some basic definitions that are needed throughout the paper.
For two functions, $f$ and $g$, where $g$ only takes strictly positive real values such that $\abs{f}/g$ is bounded, we write $f = O(g)$ or
$f\ll g$. If $\lim_{x\to \infty} \abs{f(x)}/g(x) = 0$ we write $f = o(g)$ and if $\lim_{x\to \infty} f(x)/g(x) = 1$ we write $f \sim g$.

Furthermore, we give a formal definition of the VC-dimension.
\begin{definition}
  Let $H$ be a set family (a set of sets) and $C$ a set. Their intersection is defined as the following set-family:
\begin{align*}
    H\cap C \coloneqq \setm{h\cap C}{h\in H}.
\end{align*}
We say that a set $C$ is \emph{shattered} by $H$ if $H\cap C$ contains all the subsets of $C$, i.e.:
\begin{align*}
   |H\cap C|=2^{|C|}.
\end{align*}
The \emph{VC-dimension of $H$}, which we denote by $\VC{H}$, is the largest integer $D$ such that there exists a set $C$ with cardinality $D$ that is shattered by $H$.
	\end{definition}

	We are working on the torus $\mathbb{T}$, i.e. the interval $[0,1]$ where we identify $0$ with $1$.
	Thus, an interval takes either the form $[a,b]$ for $a < b$ or $[a,b] := [0,b] \cup [a,1]$ for $a > b$.
	We define open intervals $(a,b)$ analogously.
	We can also assign each interval a \emph{length} which is given by $b-a$ and $1-a+b$ corresponding to the cases described above.
	
	This allows us to define the set $\mathcal{B}^d_{\text{per}}$ of \emph{$d$-dimensional axis-parallel boxes}, within $\mathbb{T}^{d}$ the $d$-dimensional torus, as the product of $d$ intervals on the torus i.e. $\mathcal{B}^d_{\text{per}} := \setm{[a_1, b_1] \times \ldots \times [a_d, b_d]}{0 \leq a_i \neq b_i \leq 1$ for all $i = 1,\ldots d}$.
	Furthermore, we define the subset of \emph{$d$-dimensional axis-parallel cubes} $\mathcal{C}^d_{\text{per}}$, as axis-parallel boxes where the intervals have the same length.

	It is clear that $H$ shattering a set $C$ is equivalent to $H^{c}$ shattering the same set $C$, where $H^{c}$ denotes the element-wise complement of $H$.
	
	Thus, we can also work with the element-wise complement of $\mathcal{B}^d_{\text{per}}$, in which element are of the form
	\begin{align*}
		([a_1, b_1] \times \ldots \times [a_d, b_d])^{\mathrm{c}} = \bigcup_{1\leq i \leq d} \mathbb{T}^{i-1} \times (b_i, a_i) \times \mathbb{T}^{d-i}.
	\end{align*}
	We are especially interested in the case where $b_i < a_i$ for all $i = 1,\ldots, d$.
	
	\begin{definition}
		We call $\mathbb{T}^{i-1} \times (b_i, a_i) \times \mathbb{T}^{d-i}$ a \emph{$d$-dimensional stripe anchored in dimension $i$} if $b_i < a_i$.
		Furthermore, we denote by $\mathcal{S}^d$ the set of all $d$-dimensional stripes (anchored in some dimension $i$).
		Moreover, we define its \emph{length} to be $a_i-b_i$ and denote by $\mathcal{S}^d_{l}$ the set of all $d$-dimensional stripes of length $l$.
	\end{definition}
	
	\begin{remark}

		We note that it is equivalent for stripes to be defined on $\mathbb{T}^d$ or $[0,1]^d$ as $b_i < a_i$.
	\end{remark}
	
	
	Note that in particular the complement of $\mathcal{B}^d_{\text{per}}$ contains the union of $d$-dimensional stripes anchored in different dimensions. Moreover, the complement of $\mathcal{C}^{d}_{\text{per}}$ contains the union of $d$ stripes anchored in different dimensions with equal length.

\subsection{The special subclass $\mathcal{S}^d_l$}
	This part is devoted to estimating $\VC{\mathcal{S}^d_l}$.
	\begin{proposition}\label{subclass S^d}
		We have for any $0<l<1$ and sufficiently large $d$, 
		\begin{align*}
			\log_2(d) \leq \VC{\mathcal{S}^d_l} \leq \log_2(d) + 2 \log_2(\log_2(d)) + 1.
		\end{align*}
		This shows in particular that $\VC{\mathcal{S}^d_l } \sim \log_2(d)$.
	\end{proposition}
	\begin{proof}
		To show the first inequality, it is sufficient to construct a set $C$ of $n+1$ points in dimension $2^n$ that are shattered by $S_{2^n}$.
		Therefore, we pair the subsets of $C$ as $\set{C_i, C\setminus C_i}$. Obviously, there are $2^n$ such pairs.
		Now we choose the coordinates in dimension $i$ such that the points belonging to $C_i$ have coordinates equal to $\frac{1-l}{3}$ and the points belonging to $C\setminus C_i$ have coordinates equal to $\frac{2+l}{3}$. These coordinates have distance $\frac{2l + 1}{3} > l$, so that any interval of length $l$ can contain at most one of them.
		
		Thus, we can construct for any $C' \subset C$ some $h \in S^{2^n}_l$ such that $C \cap h = C'$.
		Indeed write $C' = C_i$ or $C' = C \setminus C_i$. In any case we can choose a stripe anchored in dimension $i$ with length $l$ that contains either $\frac{1-l}{3}$ or $\frac{2l+1}{3}$ and avoids the other.
		
		It remains to show the upper bound.
		Therefore, we count the number of ways in which points can be separated with $d$-dimensional stripes.
		Let us take now $n$ points in $\mathbb{T}^d$.
		The coordinates of these points in dimension $i$ are cyclically ordered, so we can only take intervals in this order. Therefore, there are at most $(n+1)^2$ ways that a stripe anchored in dimension $i$ can separate points.
		Thus, we have in total at most $d \cdot (n+1)^2$ ways to separate points.
		Assuming that $\mathcal{S}^d_l$ shatters $n$ points in dimension $d$ gives, therefore, 
		\begin{align*}
			2^{n} \leq d \cdot (n+1)^2.
		\end{align*}
		This gives in particular for $n = \log_2(d) + 2 \log_2(\log_2(d)) + 1$,
		\begin{align*}
			d \cdot \log_2(d)^2 \cdot 2 \leq d \cdot (\log_2(d) + 2 \log_2(\log_2(d))+2)^2,
		\end{align*}
		which obviously does not hold for large enough $d$.
		This shows that for large enough $d$ we cannot shatter $\log_2(d) + 2 \log_2(\log_2(d)) + 1$ points.
	\end{proof}

\section{Upper bounds for $\VC{\mathcal{B}^d_{\mathrm{per}}}$}\label{sec_3}
Similar computations also give surprisingly accurate upper bounds for $\VC{\mathcal{B}^d_{\text{per}}}$.
We first give an easy version which is computed almost identically to the upper bound for $\VC{\mathcal{S}^d}$.
Basically the same upper bound was also given by van der Vaart and Wellner in \cite{vdVaart2009}.

	\begin{lemma}
		We have for sufficiently large $d$,
		\begin{align*}
			\VC{\mathcal{B}^d_{\text{per}}} \leq 3 d \log_2(d).
		\end{align*}
	\end{lemma}
	\begin{proof}
		We note that, similarly to the upper bound for $\VC{\mathcal{S}^d_l}$, there are at most $(n+1)^2$ different ways to intersect $n$ points when one only looks at one dimension. Therefore, there are in total at most $(n+1)^{2d}$ different ways to intersect $n$ points in $d$ dimensions, i.e.
		\begin{align}\label{eq_trivial_upper_bound}
			\abs{\setm{C \cap h}{h \in \mathcal{B}^d_{\text{per}}}} \leq (n+1)^{2d}.
		\end{align}
		Thus, assuming that $\mathcal{B}^d_{\text{per}}$ shatters $n$ points, we have
		\begin{align*}
			2^{n} \leq (n +1)^{2d}.
		\end{align*} 
		Now if one puts $n = 3 d \log_2(d)$ this gives
		\begin{align*}
			d^{3d} \leq (3 d \log_2(d) +1)^{2d},
		\end{align*}
		which does not hold for large enough $d$.
		This shows that for large enough $d$ we cannot shatter $3 d \log_2(d)$ points.
	\end{proof}

	The main over-simplification in the proof above origins from considering all the coordinates independently.
	Now we give a more involved, but also more accurate estimate.

	\begin{theorem}\label{th_VC_S}
		For sufficiently large $d$ we have 
		\begin{align*}
			\VC{\mathcal{B}^d_{\text{per}}} \leq d (\log_2(d) + 3 \log_2(\log_2(d))).
		\end{align*}
	\end{theorem}
	\begin{remark}
		The factor $3$ is not optimized and could for example be replaced by $2 + \varepsilon$ for any $\varepsilon > 0$ or equivalently by $2+ o(1)$, i.e. letting $\varepsilon$ depend on $d$.
	\end{remark}
	
	\begin{proof}
		We fix a set $C$ of $n$ points in $\mathbb{T}^d$. Given $h \in \mathcal{B}^d_{\text{per}}$, we can write $h = [a_1, b_1] \times \ldots \times [a_d, b_d]$. We define restricted versions of $h$ via
		\begin{align*}
			h_0 &:= \mathbb{T}^d\\
			h_1 &:= [a_1, b_1] \times \mathbb{T}^{d-1}\\
				&\vdots\\
			h_i &=  [a_1, b_1] \times \ldots \times [a_i, b_i] \times \mathbb{T}^{d-i}\\
				&\vdots\\
			h_d &= [a_1, b_1] \times \ldots \times [a_d, b_d].
		\end{align*}
		We have in particular $\mathbb{T}^d = h_0 \supset h_1 \supset \ldots \supset h_d = h$.
		
		This allows us to define $k(h) := (k(h, 0), k(h, 1), \ldots , k(h, d))$ via
		\begin{align*}
			k(h,0) &:= \abs{C \cap h_0} = n.\\
			k(h, i) &:= \abs{C \cap h_{i-1}} - \abs{C \cap h_i} \, \forall \, i = 1,\ldots,d.
		\end{align*}
		A simple computation gives $\abs{C \cap h} = n - k(h,1) - \ldots - k(h, d)$.
		
		Now we want to count how many different ways some $h\in \mathcal{B}^d_{\text{per}}$ can separate points of $C$ for some given $k(h)$.
		We start by fixing some $(k_1, \ldots, k_d)$ where each $k_i > 0$ and $k_1 + \ldots + k_d \leq n$ (the case where some $k_i = 0$ will be dealt with separately).
		
		We start by proving
		\begin{align*}
			\abs{\setm{C \cap h_1}{h \in \mathcal{B}^d_{\text{per}}, k(h) = (n, k_1, \ldots,k_d)}} \leq n.
		\end{align*}
		To see this inequality we look again at the coordinates of the points in $C$ in dimension $1$. 
		These are again cyclically ordered, so that there are at most $n$ different ways of separating $n - k_1$ consecutive points.
		Furthermore, we denote by 
		\begin{align*}
			\cC_1 := \setm{C \cap h_1}{h \in \mathcal{B}^d_{\text{per}}, k(h) = (n, k_1, \ldots,k_d)}.
		\end{align*}
		We find by the discussion above that $\abs{\cC_1} \leq n$ and for any $C' \in \cC_1$ we have $\abs{C'} = n-k_1$.
		
		Next we find as $h_2 \subseteq h_1$,
		\begin{align*}
			&\setm{C \cap h_2}{h \in \mathcal{B}^d_{\text{per}}, k(h) = (n, k_1, \ldots, k_d)}\\
			&\qquad =\setm{(C \cap h_1) \cap h_2}{h \in \mathcal{B}^d_{\text{per}}, k(h) = (n, k_1, \ldots, k_d)}\\
			&\qquad \subseteq \bigcup_{C' \in \cC_1} \setm{C' \cap h_2}{h \in \mathcal{B}^d_{\text{per}}, k(h) = (n, k_1, \ldots, k_d)}
		\end{align*}
		We find by the same reasoning as above that for any $C' \in \cC_1$ we have
		\begin{align*}
			\abs{ \setm{C' \cap h_2}{h \in \mathcal{B}^d_{\text{per}}, k(h) = (n, k_1, \ldots, k_d)}} \leq n-k_1.
		\end{align*}
		This shows in total that 
		\begin{align*}
			\abs{\setm{C \cap h_2}{h \in \mathcal{B}^d_{\text{per}}, k(h) = (n, k_1, \ldots, k_d)}} \leq n \cdot (n-k_1).
		\end{align*}
		
		An inductive argument proves (since $h_d = h$),
		\begin{align*}
			&\abs{\setm{C \cap h}{h \in \mathcal{B}^d_{\text{per}}, k(h) = (n, k_1, \ldots, k_d)}}\\
			&\qquad \qquad  \leq n \cdot (n-k_1) \cdots (n-k_1-\ldots-k_{d-1}).
		\end{align*}
		
		We find in total
		
		\begin{align*}
			&\abs{\setm{C \cap h}{h \in \mathcal{B}^d_{\text{per}}}}  \\ 
			& \leq \sum_{\substack{0 \leq k_1, \ldots, k_d\\ k_1 + \ldots + k_d \leq n}} \abs{\setm{C \cap h}{h \in \mathcal{B}^d_{\text{per}}, k(h) = (n, k_1, \ldots, k_d)}}.
		\end{align*}

		We already have a good upper bound for the case where $k_i > 0$ for all $i = 1,\ldots,d$ and we compute this one explicitly,
		\begin{align*}
		\sum_{\substack{0<k_{1},\ldots,k_{d}\\ k_{1}+\ldots+k_{d}\leq n}} & \abs{\setm{C \cap h}{h \in \mathcal{B}^d_{\text{per}}, k(h) = (n, k_1, \ldots, k_d)}}\\
		&\leq \sum_{\substack{0<k_{1},\ldots,k_{d}\\
k_{1}+\ldots+k_{d}\leq n
}
} n\cdot(n-k_{1})\cdots(n-k_{1}-\ldots-k_{d-1})\\
 & =\sum_{k_{1} = 1}^{n}\sum_{k_{2} = 1}^{n-k_{1}}\ldots\sum_{k_{d} = 1}^{n-k_{1}-\ldots-k_{d-1}}n\cdot(n-k_{1})\cdots(n-k_{1}-\ldots-k_{d-1})\\
 & =n\sum_{k_{1} = 1}^{n} (n-k_{1})
 \ldots\sum_{k_{d-1} = 1}^{n-k_{1}-\ldots-k_{d-2}}(n-k_{1}-\ldots-k_{d-1})\sum_{k_{d} = 1}^{n-k_{1}-\ldots-k_{d-1}}1\\
 & = n\sum_{k_{1} = 1}^{n} (n-k_{1})\ldots\sum_{k_{d-1} = 1}^{n-k_{1}-\ldots-k_{d-2}}((n-k_{1}-\ldots-k_{d-2})-k_{d-1})^{2}.
		\end{align*}

		We see that the inner most sum is actually
		\begin{align*}
			\sum_{i = 0}^{n-k_1 - \ldots - k_{d-2} - 1} i^2
		\end{align*}
		in reverse order. One can bound this sum from above by an integral
		\begin{align}\label{eq_integral}
			\sum_{i = 0}^{x-1} i^m \leq \int_0^{x} t^m dt = \frac{x^{m+1}}{m+1}.
		\end{align}
		
		Inserting \eqref{eq_integral} gives the upper bound
		\begin{align*}
			\sum_{\substack{0<k_{1},\ldots,k_{d}\\ k_{1}+\ldots+k_{d}\leq n}} & \abs{\setm{C \cap h}{h \in \mathcal{B}^d_{\text{per}}, k(h) = (n, k_1, \ldots, k_d)}}\\
				&\leq n \sum_{k_1 = 1}^{n} (n-k_{1}) \ldots \sum_{k_{d-2} = 1}^{n-k_1 - \ldots - k_{d-3}} \frac{(n-k_1-\ldots-k_{d-2})^4}{3}.
		\end{align*}
		Using the same trick as above yields
		\begin{align*}
			n &\sum_{0 < k_1 \leq n} (n-k_{1}) \ldots \sum_{0 < k_{d-2} \leq n-k_1 - \ldots - k_{d-3}} \frac{(n-k_1-\ldots-k_{d-2})^4}{3} \\
				&\leq n \sum_{k_1 =1}^{n} (n-k_{1}) \ldots \sum_{k_{d-3} = 1}^{n-k_1 - \ldots - k_{d-4}} \frac{(n-k_1-\ldots-k_{d-3})^6}{3\cdot 5}.
		\end{align*}
		
		Iterating this procedure gives in total
		\begin{align*}
			\sum_{\substack{0<k_{1},\ldots,k_{d}\\ k_{1}+\ldots+k_{d}\leq n}} & \abs{\setm{C \cap h}{h \in \mathcal{B}^d_{\text{per}}, k(h) = (n, k_1, \ldots, k_d)}}\\
				&\leq \frac{n^{2d}}{1 \cdot 3 \cdot 5 \cdots (2d-1)}\\
				&= \frac{n^{2d}}{(2d-1)!!},
		\end{align*}
		where $n!! \coloneqq n(n-2)(n-4)\cdot\ldots \cdot 1$ denotes the double factorial.
		
		Now we discuss the case where at least one of the $k_i = 0$.
		Suppose that there are $j$ indices $i_1, \ldots, i_j$ such that $k_{i_1} = \ldots = k_{i_j} = 0$.
		One sees that this is equivalent to changing $d \mapsto d-j$ and ignoring the coordinates $i_1, \ldots, i_j$, which forces the remaining $k_i$ to fulfill $k_i > 0$.
		This gives now the total estimate
		
		\begin{align*}
			\abs{\setm{C \cap h}{h \in \mathcal{B}^d_{\text{per}}}} \leq \sum_{0 \leq j \leq d} \binom{d}{j} \frac{n^{2(d-j)}}{(2(d-j)-1)!!}.
		\end{align*}
		Assuming $n \geq d$, we find
		\begin{align}
			\abs{\setm{C \cap h}{h \in \mathcal{B}^d_{\text{per}}}} \leq \sum_{0 \leq j \leq d} \binom{d}{j} \frac{n^{2d}}{(2d-1)!!} = \frac{2^d \cdot n^{2d}}{(2d-1)!!}.
		\end{align}
		Thus, we have improved on~\eqref{eq_trivial_upper_bound} by more than a factor $\frac{2^d}{(2d-1)!!}$.
		We rewrite this expression and use Stirling's formula to find
		\begin{align*}
			\frac{2^d \cdot n^{2d}}{(2d-1)!!} &= 2^d \cdot n^{2d} \cdot \frac{2^d \cdot d!}{(2d)!} \leq (4 n^2)^{d} \frac{\sqrt{2 \pi d} \rb{\frac{d}{e}}^{d} e^{1/12 d}}{\sqrt{2 \pi 2d}\rb{\frac{2d}{e}}^{2d} e^{1/(24d +1)}}\\
				&= \rb{\frac{n^2 e}{d}}^{d} \frac{e^{1/12d - 1/(24d +1)}}{\sqrt{2}} \leq \rb{\frac{n^2 e}{d}}^{d} \frac{e}{\sqrt{2}}.
		\end{align*}
		
		Thus, assuming that $\mathcal{B}^d_{\text{per}}$ shatters $n$ points, we have
		\begin{align*}
			2^{n} \leq \rb{\frac{n^2 e}{d}}^{d} \frac{e}{\sqrt{2}}.
		\end{align*} 
		Now if one puts $n = d (\log_2(d) + 3 \log_2(\log_2(d)))$ this gives
		\begin{align*}
			d^{d} \cdot \log_2(d)^{3d} \leq \rb{d (\log_2(d) + 3 \log_2(\log_2(d)))^2 e}^{d} \frac{e}{\sqrt{2}} 
		\end{align*}
		or equivalently
		\begin{align*}
			\log_2(d)^3 \leq (\log_2(d) + 3 \log_2(\log_2(d)))^2 \cdot e \cdot \rb{\frac{e}{\sqrt{2}}}^{1/d},
		\end{align*}
		which obviously does not hold for large enough $d$.
		This shows that for large enough $d$ we cannot shatter $d (\log_2(d) + 3 \log_2(\log_2(d)))$ points.
	\end{proof}
	
\section{ A lower bound for  $\VC{\mathcal{C}^d_{\mathrm{per}}}$}\label{sec_4}
To find a lower bound one naturally aims to construct some configuration of points that can be shattered by $\mathcal{C}^d_{\text{per}}$.
The main idea is to start with a set of points that can be shattered by $\mathcal{S}^k_l$ for some $k\leq d$ and map them multiple times into $\mathbb{T}^d$ such that they can be shattered by $\mathcal{C}^d_{\text{per}}$.
Therefore, we will use the following concept.

\subsection{Extraction Property}
	\begin{definition}
		We say a $c \times d$ matrix $M$ with entries in $\mathcal{A}_k = \set{a_1, \ldots, a_k}$ has the \emph{$k$-extraction property} if for any $b_1, \ldots, b_c \in \mathcal{A}_k$ there exist $j_1, \ldots, j_c \in \set{1,\ldots, d}$ which are pairwise different, such that $M_{i, j_i} = b_i$ for all $i$.
	\end{definition}
	
	\begin{example}
		We show that the following $c \times (c+1)$ matrix $M$ has the $2$-extraction property.
		$M$ is given as follows,
		\begin{align*}
			M := \begin{pmatrix}
				a_1 & a_2 & a_1 & a_1 & \ldots & a_1\\
				a_1 & a_1 & a_2 & a_1 & \ldots & a_1\\
				a_1 & a_1 & a_1 & a_2 & \ldots & a_1\\
				&\vdots&&& \ddots&\vdots\\
				a_1 & a_1 & a_1 & a_1 & \ldots & a_2
			\end{pmatrix}.
		\end{align*}
		We see that $a_2$ appears in every row exactly once and always in different columns. 
		Let us take now any word $b_1, \ldots, b_c \in \set{a_1, a_2}$.
		We can choose $j_i = i+1$ for each $i$ with $b_i = a_2$. The remaining $j_i$ can be simply chosen by picking the first column, that has not yet been used, for each row that contains $a_1$.
	\end{example}
	
	Now we show how the $k$-extraction property can be used to find lower bounds for $\VC{\mathcal{C}^d_{\text{per}}}$.
	In particular, it allows to transfer lower bounds of $\VC{\mathcal{S}^k_l}$ to lower bounds for $\VC{\mathcal{C}^d_{\text{per}}}$.
	This will be the main tool for the proof of our lower bounds of $\VC{\mathcal{C}^d_{\text{per}}}$.
	\begin{proposition}\label{pr_lifting_VC}
		Suppose there exists a $c \times d$ matrix with the $k$-extraction property.
		Then for any $0 < l < 1$,
		\begin{align}
			\VC{\mathcal{C}^d_{\text{per}}} \geq c \cdot \VC{\mathcal{S}^k_l} \geq c \cdot \log_2(k).
		\end{align}
	\end{proposition}
	\begin{remark}
		The construction ensures that we only need cubes with length $1-\frac{l}{c+1}$ or equivalently stripes with length $\frac{l}{c+1}$.
	\end{remark}
	\begin{proof}
		Let us denote $u = \VC{\mathcal{S}^k_l}$ and pick a set of $u$ points $X:=\set{x(1), \ldots, x(u)} \subseteq\mathbb{T}^{k}$ that is shattered by $\mathcal{S}^k_l$.
		We write $x(p) = (x(p)_1, \ldots, x(p)_k)$.
		Furthermore, we have a $c \times d$ matrix $M$ with entries in $\mathcal{A}_k = \set{a_1, \ldots, a_k}$ fulfilling the $k$-extraction property.
		This allows us to define new points $y(i;j) \in \mathbb{T}^{d}$, where $1 \leq i \leq c, 1 \leq j \leq u$.
		We write them as $y(i;j) = (y(i;j)_1, \ldots, y(i;j)_{d})$.
		We set for $1\leq n \leq d$,
		\begin{align}
			y(i;j)_n := \frac{i-1}{c+1} + \frac{x(j)_{\ell}}{c+1},
		\end{align}
		where $\ell$ is the unique integer such that $M_{i, n} = a_{\ell}$.
		
		We see in particular that $y(i; j) \in (\frac{i-1}{c+1}, \frac{i}{c+1})^d$ for all $i,j$.
		Thus, the points are grouped into $r$ sets of size $u$, all of which are contained in $(0,\frac{r}{c+1})^d$.
		
		It remains to show that we can shatter these points by $\mathcal{C}^d_{\text{per}}$.
		As already mentioned earlier, we work with the complement of $\mathcal{C}^d_{\text{per}}$ instead, which contains the set of unions of $d$ stripes of equal length.
		
		We take now an arbitrary subset $C'$ of $C := \setm{y(i;j)}{1\leq i \leq c, 1 \leq j \leq u}$.
		Our goal is to find a union of $d$ stripes of length $\frac{l}{c+1}$ such that its intersection with $C$ is $C'$.
		
		We define for each $1\leq i \leq c$ a set $C_i := \setm{y(i;j)}{1\le j\le u, y(i;j) \in C'}$ and $X_i := \setm{x(j)}{1\le j\le u, y(i,j)\in C'}$.
		As we assumed that we can shatter $X$ by $\mathcal{S}^k_l$, there exists a $k$-dimensional stripe of length $l$ anchored in some dimension $m_i$ such that intersection with $X$ is $X_i$. We denote it by $\mathbb{T}^{m_i-1} \times (\alpha_i, \beta_i) \times \mathbb{T}^{k-m_i}$.
		Finally we define $b_i := a_{m_i}$ for $1\leq i \leq c$.
		
		As $M$ has the $k$-extraction property we find some pairwise different $n_1, \ldots, n_c \in \set{1, \ldots, d}$ such that $M_{i, n_i} = b_i$ for all $i$.
		Now for a given $1 \leq i \leq c$ we define a stripe anchored in dimension $n_i$ by $\mathbb{T}^{n_i-1} \times (\frac{i-1 + \alpha_i}{c+1}, \frac{i-1 + \beta_i}{c+1}) \times \mathbb{T}^{d-n_i}$. It is clear that this stripe has length $\frac{l}{c+1}$ and we claim that its intersection with $C$ is $C_i$.

Indeed it cannot contain any points of the form $y(i',j)$ for $i' \neq i$ as they belong to $(\frac{i'-1}{c+1}, \frac{i'}{c+1})^d$ which is disjoint from $\mathbb{T}^{n_i-1} \times (\frac{i-1}{c+1}, \frac{i}{c+1}) \times \mathbb{T}^{d-n_i}$ which contains the stripe.
		Thus, we only need to consider $y(i;j)$ for $1\leq j \leq u$. By the definition of the stripe, we are only interested in the $n_i$-th coordinate.
		We have by definition
		\begin{align*}
			y(i;j)_{n_i} = \frac{i-1}{c+1} + \frac{x(j)_{\ell}}{c+1},
		\end{align*}
		where $\ell$ is given by $M_{i, n_i} = a_{\ell}$. By the extraction choice we have that $M_{i, n_i} = b_i = a_{m_i}$.
		This shows that $\ell = m_i$.
		We have in total that $y(i;j) \in \mathbb{T}^{n_i-1} \times (\frac{i-1 + \alpha_i}{c+1}, \frac{i-1 + \beta_i}{c+1}) \times \mathbb{T}^{d-n_i}$ if and only if $x(j)_{m_i} \in (\alpha_i, \beta_i)$, which is by definition the case if and only if $x(j) \in X_i$ or equivalently $y(i;j) \in C'$.
		
		Thus, as $C'=C_1\cup\dots\cup C_c$, we can cover $C'$ with stripes with length $\frac{l}{c+1}$ anchored in pairwise different dimensions.
		Now it just remains to choose the stripes $\mathbb{T}^{i-1} \times (\frac{c}{c+1}, \frac{c+l}{c+1}) \times \mathbb{T}^{d-i}$ for the remaining dimensions which contain no points.
		The second inequality is simply an application of Theorem~\ref{th_VC_S}.
	\end{proof}
	
	\subsection{An equivalent definition for the extraction property and a random construction}
	This equivalent definition will be particularly useful for the generation of matrices with the extraction property.

\begin{proposition}\label{P:ProprieteExtraction}
A $c \times d$ matrix $M$ with entries in $\mathcal{A}_k = \set{a_1, \ldots, a_k}$ does not satisfy the $k$-extraction property if and only if there exist $U\subseteq\set{1,\dots,c}$, and $V\subseteq\set{1,\dots,d}$ of cardinality $\abs{U}-1$ such that for all $u\in U$ there is $b\in\mathcal{A}_k$ such that for all $1\le j\le d$ if $M_{u,j}=b$ then $j\in V$.
\end{proposition}

\begin{proof}
The matrix $M$ has not the $k$-extraction property if and only if there exists some $b_1, \ldots, b_c \in \set{a_1, \ldots, a_k}^c$, such that there exist no pairwise different $n_1, \ldots, n_c$ such that $M_{\ell, n_{\ell}} = b_{\ell}$ for all $1\leq \ell \leq c$.

We assign now to any $1\leq \ell\leq c$ a set $X_{\ell} := \setm{1\leq j \leq d}{M_{\ell, j} = b_{\ell}} \subset \set{1,\ldots, d}$.
		With this notation, we have the equivalent statement that there exists no injective function $f\colon \set{1,\ldots, c} \to \set{1,\ldots, d}$ such that $f(\ell) \in X_{\ell}$.
		This is by Hall's Marriage Theorem equivalent to the violation of the \emph{marriage condition} which states that for any $W \subset \set{1,\ldots, c}$ we have $\abs{W} \leq \abs{\bigcup_{\ell \in W} X_{\ell}}$.

If $W$ fails the marriage condition, then with $U=W$, we have $\abs{\bigcup_{\ell \in W} X_{\ell}}<\abs{U}$,
hence we can pick $V\supseteq \bigcup_{\ell \in W} X_{\ell}$ of cardinality $\abs{U}-1$. Then $U,V$ satisfies the required condition of the proposition. Reciprocally if $U,V$ satisfies the condition then $U$ fails the marriage condition.
\end{proof}

\begin{remark}
If $U,V$ satisfies the property of Proposition~\ref{P:ProprieteExtraction} then we say that $U,V$ \emph{witness the failure of the $k$-extraction property}.
\end{remark}

It remains to find good $c \times d$ matrices with the $k$-extraction property, i.e. for given $d$ we want $c \log_2(k)$ to be as large as possible. The next lemma gives a non-constructive argument ensuring the existence of matrices with the $k$-extraction property.

\begin{lemma}\label{le_construct_k_extraction}
Let $q\in\Q$ with $q>1$. Let $k\ge 1$ and $m\ge 1$ be integers. Assume that $qm$ is an integer, and 
\begin{equation}\label{Ext-Req}
(q-\frac{q}{k})^{qm}>qm^2k^3\, . 
\end{equation}

Set $c=mk$, and $d=qmk$. There exists a $c\times d$ matrix with the $k$-extraction property.
\end{lemma}

\begin{proof}
Set $\mathcal{A}_k=\set{a_1,\dots,a_k}$. We say that a word of length $d=qmk$ is \emph{balanced} if each symbol of $\mathcal{A}_k$ appears exactly $qm$ times. We call a $c\times d$ matrix \emph{balanced} if each row corresponds to a balanced word. The number of balanced words of length $d$ is
\begin{equation}\label{E:A}
A=\frac{d!}{ (qm)!^k}\,.
\end{equation}
Therefore the number of balanced $c\times d$ matrices is
\begin{equation}\label{E:T}
T=A^c=\left(\frac{d!}{ (qm)!^k}\right)^c\,.
\end{equation}
We shall compute an upper bound for the number $B$ of balanced $c\times d$ matrices not having the $k$-extraction property, by looking at the ones failing the equivalent property seen in Proposition~\ref{P:ProprieteExtraction}.

Let $i\le c$, $U\subseteq \set{1,\dots,c}$ be of cardinality $i$ and $V\subseteq \set{1,\dots,d}$ be of cardinality $i-1$. Let $b\in\mathcal{A}_k$. Denote by $E_i$ the number of balanced words of length $d$ such that the symbol $b$ can only appear at positions in $V$. We have
\begin{equation*}
E_i=\binom{i-1}{qm}\frac{(qmk-qm)!}{(qm)!^{k-1}}\,,
\end{equation*}
as we can start by picking the $qm$ positions where $b$ appears, among the $i-1$ possible positions, then pick all other symbols (among the $qmk-qm$ positions left).

Denote by $F_i$ the number of balanced words of length $d$ such that there is a symbol $b$ that can only appear at positions in $V$. We have
\begin{equation}\label{E:Fi}
F_i\le kE_i = k\binom{i-1}{qm}\frac{(qmk-qm)!}{(qm)!^{k-1}}\,.
\end{equation}
Indeed the set of words counted by $F_i$ is the union for $b\in\mathcal{A}_k$ of the set of words counted by $E_i$.

The number $G_i$ of balanced $c\times d$ matrices, such that $U,V$ witness the failure of the $k$-extraction property is therefore
\[
G_i = F_i^i A^{c-i}\,.
\]
Indeed, for each line, whose index is in $U$, we have to pick a word with at least one symbol appearing only at positions within $V$ (this is counted by $F_i$). For other lines we can pick any balanced word (counted by $A$). Finally denote by $H_i$ the number of balanced $c\times d$ matrices, such that there are $U\subseteq \set{1,\dots,c}$ of cardinality $i$ and $V\subseteq \set{1,\dots,d}$ of cardinality $i-1$ that witness the failure of the extraction property. We have
\begin{equation}\label{E:Hi}
H_i\le \binom{c}{i}\binom{d}{i-1} G_i = \binom{c}{i}\binom{d}{i-1}F_i^i A^{c-i}\,.
\end{equation}
Once again, $H_i$ is counting the elements of a union of $\binom{c}{i}\binom{d}{i-1}$ sets of cardinality $G_i$.

It follows from Proposition~\ref{P:ProprieteExtraction} that the number $B$ of balanced $c\times d$ matrices failing the $k$-extraction property satisfies
\begin{equation}\label{E:B}
B\le \sum_{i=1}^c H_i\,.
\end{equation}

From \eqref{E:T}, \eqref{E:Hi}, and the upper bounds $\binom{c}{i}\le c^i$ and $\binom{d}{i-1}\le d^{i-1}$ we have
\begin{equation}\label{E:THi}
\frac{T}{H_i} \ge \frac{A^c}{\binom{c}{i}\binom{d}{i-1}F_i^i A^{c-i}} =  \frac{A^i}{\binom{c}{i}\binom{d}{i-1}F_i^i} \ge d\frac{A^i}{ c^i d^i F_i^i} = d \left(\frac{A}{ c d F_i}\right)^i
\end{equation}
Moreover, from \eqref{E:A} and \eqref{E:Fi} we obtain
\begin{equation}\label{E:ArsFi}
\frac{A}{ c d F_i} \ge  \frac{\frac{d!}{ (qm)!^k} }{cd k\binom{i-1}{qm}\frac{(qmk-qm)!}{(qm)!^{k-1}}} = \frac{d! }{cd k(qm)! \binom{i-1}{qm}(qmk-qm)!}
\end{equation}
We observe that
\begin{equation}\label{E:I1}
 \frac{d!}{(qmk-qm)!} = \frac{(qmk)! }{(qmk-qm)!}\ge (qmk-qm)^{qm}\,,
\end{equation}
and
\begin{equation}\label{E:I2}
(qm)! \binom{i-1}{qm} \le (i-1)^{qm} < c^{qm}= (mk)^{qm}\,.
\end{equation}
Note that by assumption we have $\left(q-\frac{q}{k}\right)^{qm}>qm^2k^3=cdk$. Therefore from \eqref{E:ArsFi},\eqref{E:I1}, and \eqref{E:I2} we obtain
\begin{equation}\label{E:ArsFi2}
\frac{A}{ c d F_i} \ge  \frac{(qmk-qm)^{qm}}{cdk(mk)^{qm}}= \frac{1}{cdk} \left(q-\frac{q}{k}\right)^{qm} >1
\end{equation}
Thus from \eqref{E:THi} and \eqref{E:ArsFi2} we obtain $\frac{T}{H_i}>d=qc$. Hence it follows from \eqref{E:B} that
\[
\frac{B}{T} \le\sum_{i=1}^{c} \frac{H_i}{T}  < \sum_{i=1}^{c} \frac{1}{qc}= \frac{1}{q}<1\,.
\] 
Therefore $T>B$, hence there is a balanced matrix which does not fail the $k$-extraction property.
\end{proof}

\subsection{Finding a lower bound}
Now we use Lemma~\ref{le_construct_k_extraction} to give explicit lower bounds for $\VC{\mathcal{C}^d_{\text{per}}}$.
Thus we have to find some $q \in \Q, m \in \N, k \in \N$ which satisfy \eqref{Ext-Req} where all of these variables will depend on $d$.

We start by choosing
\begin{align}
	q := 1 + \frac{1}{f(d)} = \frac{1 + f(d)}{f(d)} = \frac{1}{1 - \frac{1}{1+ f(d)}},
\end{align}
for some function $f: \N \to \N$ that is slowly growing and unbounded.
We will choose $m$ and $k$ satisfying
\begin{align}
	m &= 24 f(d) \floor{\log_2 (d)}\\
	k &= \floor{\frac{d}{mq}} = \floor{\frac{d}{24 (f(d) +1) \floor{\log_2 (d)}}}.
\end{align}
We claim that this choice satisfies \eqref{Ext-Req} as long as
\begin{align}\label{eq_condition_f}
	\frac{d}{\floor{\log_2(d)}} > 48 (f(d) +2)^2.
\end{align}
Indeed this assures that $q \leq 2, m < d, k <d$.
We, furthermore, note that \eqref{eq_condition_f} implies $k > 2q/(q-1)$, or equivalently $q - q/k > (q+1)/2$.
Thus, we have
\begin{align*}
	\rb{q - \frac{q}{k}}^{qm} &> \rb{\frac{q+1}{2}}^{m} = \rb{1 + \frac{1}{2 f(d)}}^{m} = e^{m \log(1 + \frac{1}{2f(d)})} \geq e^{\frac{m}{4 f(d)}}\\
		& = e^{6 \floor{\log_2 d}} \geq \rb{\frac{e^{\log_2(d)}}{e}}^6 \geq \rb{\frac{d}{e}}^6 \geq q m^2 k^3,
\end{align*}
for large enough $d$.

Thus, for a fixed $d$ we have $d'\coloneqq qmk \leq d$ and we can apply Lemma~\ref{le_construct_k_extraction} and Proposition~\ref{pr_lifting_VC} to find,

\begin{align*}
	\VC{\mathcal{C}^{d'}_{\text{per}}} &\geq mk \log_2(k) = \frac{d'}{q} \log_2\rb{\frac{d'}{qm}}\\
		& = d' \rb{1 - \frac{1}{f(d)+1}} \rb{\log_2(d') - \log_2(q) - \log_2(m)}\\
		& = d' \rb{1 - \frac{1}{f(d)+1}} \cdot \\
		& \phantom{=} \rb{\log_2(d') - \log_2(24) - \log_2(f(d)+1) - \log_2(\floor{\log_2(d)})}.
\end{align*}

We already see that following this strategy we cannot do any better than $\VC{\mathcal{C}^{d'}_{\text{per}}} \geq d' (\log_2(d')) - \log_2(\log_2(d))$.
Let us choose $f(d) = \floor{\log_2(d)}$.
We see immediately that \eqref{eq_condition_f} is satisfied for large enough $d$.
Furthermore our estimate for $\VC{\mathcal{C}^{d'}_{\text{per}}}$ becomes now
\begin{align*}
	\VC{\mathcal{C}^{d'}_{\text{per}}} &\geq d' \rb{1 - \frac{1}{2 \log_2(d)}} \rb{\log_2(d') - \log_2(48) - 2 \log_2(\floor{\log_2(d)})}\\
		&\geq d' (\log_2(d') - 3 \log_2(\log_2(d)))
\end{align*}


for large enough $d$.
We finally also have $d' \geq d - qm = d - 24 (\floor{\log_2(d)}^2+ \floor{\log_2(d)}) $, which gives in particular $\log_2(d') \geq \log_2(d) - 1$ for large enough $d$.

This gives our final estimate
\begin{align*}
	\VC{\mathcal{C}^d_{\text{per}}} &\geq \VC{\mathcal{C}^{d'}_{\text{per}}} \geq d' (\log_2(d') - 3 \log_2(\log_2(d)))\\
		&\geq  d (\log_2(d') - 3 \log_2(\log_2(d))) - 24 (\log_2(d))^3\\
		&\geq d (\log_2(d) - 4 \log_2(\log_2(d))).
\end{align*}

Thus, we have shown the lower bound for $\VC{\mathcal{C}^d_{\text{per}}}$ stated in Theorem~\ref{th_main}.

\begin{remark}
	The lower bound can definitely be improved, but using this method we cannot do better than $d(\log_2(d) - \log_2(\log_2(d))$.
\end{remark}

%
%


\section*{Open Questions}
As folklore it is said that intersection-closed hypothesis set indicates a rather easy structure and thus, implies linear growth of the VC-dimension for the higher dimensional product of this set. 
As this is not given for boxes on the Torus this is one of the first indication for a non-linear growth of the VC-dimension.
Considering the result we were able to obtain here opens the question if the maximal amount of disjoint sets obtained by intersecting two different hypothesis has any influence on the growth of the VC-dimension. Further examples like boxes on the torus may be of high interest here.

Another interesting problem arises when one looks into the question for cubes with a fixed edge length. Looking at the main Theorem~\ref{th_main} one actually sees that for a fixed dimension all the cubes in the construction are of the same width, let's call this $r_d$. Looking at $\mathcal{C}^d_{\text{per}}(r)$, i.e. all cubes of width $r$ on the torus, one finds that $\VC{\mathcal{C}^d_{\text{per}}(r_1)}-1  \leq \VC{\mathcal{C}^d_{\text{per}}(r_2)}$ for all $r_1\leq r_2$. Indeed, let us assume one has $s$ points that can be shattered by $\mathcal{C}^d_{\text{per}}(r_1)$. 
As both the points and the periodic boxes (and cubes) can be freely translated on the torus, we can assume that one of the points is in the origin.
By labeling this point positively, we see that the remaining $s-1$ points are also shattered by $(\mathcal{C}^d_{\text{per}}(r_1))^c$, i.e. the union of stripes anchored in different dimensions, of the form $\mathbb{T}^{i-1} \times (b_i, b_i + 1-r_1) \times \mathbb{T}^{d-i}$, where $0 \leq b_i, b_i+1-r_1 \leq 1$.
Now we can view the $s-1$ points as elements of $[0,1]^d$ and scale them by a factor $\frac{1-r_2}{1-r_1} < 1$.
Thus, these $s-1$ points are still shattered by the scaled union of stripes, which have now length $1-r_2$. In particular, these scaled $s-1$ points are shattered by $(\mathcal{C}^d_{\text{per}}(r_2))^c$ proving the claim.
Thus, in general one could expect that the VC-Dimension of cubes is non-decreasing with the radius $r$.
Furthermore, following our construction, we have

\[
\VC{\mathcal{C}^d_{\text{per}}(r)}= d \log_2(d) + O \left( d \log \left(\log \left( d \right) \right) \right) \text{ for all } \frac{d-1}{d}\leq r < 1.
\]

It is not clear how this will work out for $0< r < \frac{d-1}{d}$. It is easy to check that for $r \leq \frac{1}{2}$ the VC-dimension is at most the same as in the non-periodic case.
So the remaining case of  $\frac{1}{2}< r < \frac{d-1}{d}$ is still open and by far not clear what kind of transition happens between being of linear growth and being of quasilinear growth.

\section*{Acknowledgements}
The authors want to thank Manfred Scheucher for providing computational results for lower dimensions with the help of a computer cluster of the TU Berlin. Furthermore, thanks to, again, Manfred Scheucher and Stefan Lendl for helpful discussions on this problem and thanks to Gabriel Conant for pointing towards the questions about cubes.

\bibliography{VC-bib}
\bibliographystyle{abbrv}

\vspace{23bp}

\end{document}